%% file: index-paper.tex
\documentclass[10pt,reqno]{amsart}

\input{research-notes-preamble.tex}
\usepackage[export]{adjustbox}

\usepackage{bm}

\DeclareMathOperator{\Br}{Br}

\newcommand{\para}[1]{\paragraph{\color{Black}\textbf{#1}}}

\usepackage[mode=buildnew]{standalone}

\newcommand{\connsum}{\mathbin{\#}}
\newcommand{\fibsum}{\mathbin{\#_F}}

\newif\ifhideexers

\ifhideexers
\usepackage{environ}
\NewEnviron{hide}{}

\fi

\graphicspath{ {resources} }

\newcommand{\VARTitle}{How large is the braid monodromy of low-genus Lefschetz fibrations?} 
\newcommand{\VARAuthor}{Faye Jackson}

\title{\vspace*{-1.75cm}\VARTitle} 
\author{\vspace{-0.25cm}\VARAuthor}
\date{}
\definecolor{darkpastelpurple}{rgb}{0.59, 0.44, 0.84}
\definecolor{mauve}{rgb}{0.88, 0.69, 1.0}

\renewcommand{\theequation}{\arabic{section}.\arabic{equation}}

\begin{document}

\hypersetup{linkcolor=lapislazuli}

\begin{abstract} 
	Given a genus $g$ smooth Lefschetz fibration $\pi : M \to S^2$ with singular locus $\Delta \subseteq S^2$, we describe the subgroup $\Br(\pi)$ of the spherical braid group $\Mod(S^2,\Delta)$ consisting of braids admitting a lift to a fiber-preserving diffeomorphism of $M$. We develop general methods for showing that the index $[\Mod(S^2,\Delta) : \Br(\pi)]$ is infinite. As an application of our methods, we prove that $[\Mod(S^2,\Delta) : \Br(\pi)] = \infty$ when $g = 1$, when $\pi$ is expressible as a self-fiber sum when $g \geq 2$, or when $\pi$ is a holomorphic genus $g = 2$ Lefschetz fibration whose vanishing cycles are nonseparating. In the genus $g = 1$ case, we relate the subgroup $\Br(\pi)$ to the action of $\Mod(S^2,\Delta)$ on the $\SL_2$-character variety for $S^2 \setminus \Delta$ and provide an alternate proof of the first application via recent work of Lam--Landesman--Litt \cite{lam-lan-litt}.
\end{abstract}

\thispagestyle{empty}
\maketitle
\setcounter{tocdepth}{2}
\microtypesetup{protrusion=false}
\microtypesetup{protrusion=true}

\lhead{\VARAuthor}
\chead{\VARTitle}
\rhead{}
\setstretch{1.2}

\section{Introduction}


\subfile{intro.tex}

\vspace{0.5\baselineskip}
\para{Code/Data Availability}\label{code:data}

A program in Sage/Python was implemented to do the necessary computations in \Cref{subsec:finite-index}. This program can be obtained from GitHub at
\begin{align*}
	\text{\href{https://github.com/FayeAlephNil/solomon}{https://github.com/FayeAlephNil/solomon}} \tag{$\dagger\dagger$}\label{eq:code}
\end{align*}
or, upon request, from the author.

\vspace{0.5\baselineskip}
\para{Acknowledgements}

I thank my advisor Benson Farb for his generosity with his time, feedback, and support. Jointly, I thank Benson Farb and Eduard Looijenga for spending time explaining their work to me. I also thank Aaron Landesman and Daniel Litt for helpful conversations concerning their previous work on character varieties. I appreciate Seraphina Lee and Carlos Servan explaining their work on homologically distinct Lefschetz fibrations, which allowed me to simplify the proof of \Cref{thm:inf-index}. Thank you to Nick Salter and Dan Margalit, whose comments significantly improved the introduction. Finally, I thank the members of the geometry and topology working group for their time and interest in my project, particularly Ethan Pesikoff and Zhong Zhang. The author was partially supported by an NSF Graduate Research Fellowship under NSF grant no. 2140001.

\section{The Hurwitz Action and \texorpdfstring{$\Br(\pi)$}{Br(pi)}}\label{sec:prelim}

Let $\pi : M \to S$ be a genus $g$ Lefschetz fibration\footnote{For a survey of the theory of Lefschetz fibrations we refer the reader to \cite{endo}.} over a connected oriented surface $S$ with at most one boundary component and with no singular values along the boundary $\partial S$. We assume throughout that each singular fiber of $\pi$ has one critical point i.e., that the singularities are simple. Let $\Delta_\pi \subseteq S \setminus \partial S$ denote the set of singular values of $\pi$. For convenience, when the fibration is clear from context, we will omit $\pi$ from the notation and simply write $\Delta$.  As in \cite{fl-smooth-mw} we let,
\begin{align}
	\Diff^+(\pi) &\coloneqq \{F \in \Diff^+(M) \st \rest{F}{\pi^{-1}(\partial S)} = \Id, F \text{ sends fibers of } \pi \text{ to fibers of } \pi\} \\
	\Mod(\pi) &\coloneqq \pi_0\Diff^+(\pi). \label{defn:mod-pi}
\end{align}
There is a natural map $\Mod(\pi) \to \Mod(S, \Delta)$ given by tracking where $F$ takes each fiber, notably, singular fibers must be sent to singular fibers. The image of $\Mod(\pi) \to \Mod(S, \Delta)$ consists precisely of those mapping classes which lift to fiber-preserving diffeomorphisms of $M$ acting as the identity on $\pi^{-1}(\partial S)$. We call such mapping classes \textit{liftable mapping classes}, and denote the subgroup of liftable mapping classes by
\begin{align}
	\Mod_\pi \coloneqq \im(\Mod(\pi) \to \Mod(S, \Delta)) \label{defn:mod-sub-pi}
\end{align}
When $S = S^2$ or $S = D^2$, we adopt the notation of the introduction
\begin{align}
	\Br(\pi) \coloneqq \Mod_\pi = \im(\Mod(\pi) \to \Mod(S^2, \Delta)) < \Mod(S^2,\Delta) \label{defn:braid-pi} \\
	\Br(\pi) \coloneqq \Mod_\pi = \im(\Mod(\pi) \to \Mod(D^2, \Delta)) < B_n,
\end{align}
where $n = \abs{\Delta}$ is the number of singular fibers. In this way we think of $\Br(\pi)$ as a subgroup of the (spherical) braid group which is associated to $\pi$.

We now give a description of the liftable mapping classes in terms of the monodromy of the Lefschetz fibration, essentially due to Moishezon in genus $g = 1$ and Matsumuto in genus $g \geq 2$. To do so, when $\partial S \neq \0$, let $b \in \partial S$ (resp. $b \in S\setminus \Delta$ if $\partial S = \0$), then let
\begin{align}
	\phi_\pi : \pi_1(S \setminus \Delta, b) \to \Mod(\pi^{-1}(b)) \cong \Mod(\Sigma_g) \label{defn:mon-rep-prelim}
\end{align}
be the monodromy representation of $\pi$, choosing some identification $\pi^{-1}(b)$ with the standard $\Sigma_g$ once and for all. For later use, recall that given an arc $\ell$ from $b$ to the image $p \in \Delta$ of some singular fiber there is an associated \textit{vanishing cycle} $\alpha \subseteq \pi^{-1}(b)$, which can be represented by a simple closed curve. Furthermore, if $\gamma$ is obtained by adjoining a small counterclockwise loop about $p$ to $\ell$, then $\phi_\pi(\gamma) = T_\alpha$, the Dehn twist about $\alpha$ (see \cite[Section 2.2]{endo}). Hence, over $S = S^2$, the monodromy representation takes the form of a factorization of the identity
\begin{align*}
	T_{\alpha_1}\cdots T_{\alpha_n} = \Id
\end{align*}
in the mapping class group $\Mod(\Sigma_g)$. The tuple $(T_{\alpha_1},\ldots,T_{\alpha_n})$ is called the \textit{monodromy factorization} associated to $\pi$.

When $\partial S \neq \0$, we find that $\Mod(S, \Delta)$ acts on $\pi_1(S \setminus \Delta, b)$, and hence on the set of all such representations by precomposition. When $\partial S = \0$, the action of $\Mod(S,\Delta)$ on $\pi_1$ is merely an outer action, and so it acts on the set of all such representations up to conjugacy. For $S = S^2$, the action by precomposition is the familiar \textit{Hurwitz action} from the theory of Lefschetz fibrations. Equipped with this notation, we state the theorems of Moishezon/Matsumuto.
\begin{theorem}[Moishezon {\cite{moishezon}}, Matsumuto {\cite{matsumuto}}]\label{thm:moishezon-matsumuto}
	Let $\pi : M \to S$ be a Lefschetz fibration of genus $g$ over a connected oriented base $S$, with singular locus $\Delta \subseteq S$. Suppose that $\partial S$ is nonempty (resp. $\partial S = \0$) and has a single component, then $\Mod_{\pi}$ is the stabilizer of $\phi_\pi$ (resp. up to conjugation) under the action of $\Mod(S,\Delta)$.
\end{theorem}

\begin{remark}
	Originally, Moishezon and Matsumuto's theorems are phrased as a classification result, i.e., that a Lefschetz fibration of genus $g$ is classified by its monodromy factorization (equivalently, its monodromy representation). The above statement is a maps version of their results, which follows immediately from their proofs. See \cite[Theorem 3.3]{endo} for a similar statement in the literature.
\end{remark}

Using \Cref{thm:moishezon-matsumuto}, Moishezon was able to classify Lefschetz fibrations in genus 1 by classifying homomorphisms $\pi_1(S^2 \setminus \Delta) \to \Mod(\Sigma_1) \cong \SL_2\ZZ$ which take simple loops about the punctures to Dehn twists. 
\begin{theorem}[Moishezon {\cite[Theorem 9]{moishezon}}]\label{thm:moishezon}
	Any nontrivial Lefschetz fibration $\pi : M \to S^2$ of genus one has $12d$ singular fibers for some integer $d \geq 1$. Furthermore, the number of singular fibers determines $\pi$ in the following sense: given genus one fibrations $\pi, \pi' : M,M' \to S^2$ with the same number of singular fibers, there are diffeomorphisms $F,f$ making the following diagram commute
	\begin{center}
		\begin{tikzcd}
			M \ar[d] \ar[r,"F"] & M' \ar[d] \\
			(S^2,\Delta_\pi) \ar[r,"f"] & (S^2, \Delta_{\pi'}).
		\end{tikzcd}
	\end{center}
\end{theorem}
In general, following \cite[Definition 3.1]{endo}, we call two Lefschetz fibrations $\pi,\pi' : M,M' \to S$ of genus $g$ \textit{weakly isomorphic} provided there exists such a pair $(F,f)$ of diffeomorphisms so that
\begin{center}
	\begin{tikzcd}
		M \ar[d] \ar[r,"F"] & M' \ar[d] \\
		(S,\Delta_\pi) \ar[r,"f"] & (S, \Delta_{\pi'})
	\end{tikzcd}
\end{center}
commutes. Changing $\pi$ by a weak isomorphism changes $\Mod_\pi$ by conjugation in $\Mod(S,\Delta_\pi)$. Furthermore $\Mod(\pi)$ consists of the isotopy classes of weak isomorphisms from $\pi$ to itself.

Equipped with \Cref{thm:moishezon-matsumuto,thm:moishezon}, we now deduce \Cref{thm:inf-index} from Lam-Landesman-Litt's work on classifying finite mapping class group orbits of rank 2 local systems. Because the action of $\Mod(\Sigma_1)$ on the homology $H_1(\Sigma_1, \ZZ)$ provides an isomorphism $\Mod(\Sigma_1) \xrightarrow{\cong} \SL_2(\ZZ)$, we find that for $\pi$ a genus one Lefschetz fibration over the sphere with singular locus $\Delta$, $\Br(\pi) = \Mod_\pi$ is the stabilizer of the conjugacy class of the monodromy representation $\phi_\pi : \pi_1(S^2 \setminus \Delta) \to \SL_2\ZZ$ under the Hurwitz action. Let
\begin{align*}
	Y_{\Delta}(\ZZ) &\coloneqq \Hom(\pi_1(S^2 \setminus \Delta), \SL_2\ZZ)/\SL_2\ZZ, \\
	Y_{\Delta}(\CC) &\coloneqq \Hom(\pi_1(S^2 \setminus \Delta), \SL_2\CC)/\SL_2\CC,
\end{align*}
where $\SL_2\ZZ$ and $\SL_2\CC$ act by conjugation. $Y_\Delta(\CC)$ parameterizes rank two local systems on $S^2 \setminus \Delta$ with trivial determinant bundle, and the GIT quotient
\begin{align*}
	\Hom(\pi_1(S^2 \setminus \Delta), \SL_2\CC)\sslash \SL_2\CC
\end{align*}
is the character variety for such local systems. By \Cref{thm:moishezon-matsumuto}, $\Br(\pi)$ has infinite index if and only if the orbit of $[\phi_\pi] \in Y_\Delta(\ZZ)$ under $\Mod(S^2, \Delta)$ is infinite. Because the natural map $Y_\Delta(\ZZ) \to Y_\Delta(\CC)$ is $\Mod(S^2, \Delta)$-equivariant, it follows immediately that $\Br(\pi)$ has infinite index inside $\Mod(S^2, \Delta)$ provided that the $\CC$-linear local system has infinite mapping class group orbit. Because the monodromy about a loop $\gamma$ surrounding a single puncture is given by the Dehn twist about the vanishing cycle associated to $\gamma$, which has infinite order, we can apply \cite[Corollary 1.1.8]{lam-lan-litt}, which classifies such rank two local systems of finite mapping class group orbit. In particular, all rank two local systems over the sphere with infinite order monodromy about more than $6$ punctures have infinite mapping class group orbit. Hence Part (1) of \Cref{thm:inf-index} holds, as any genus one Lefschetz fibration has $12d > 6$ singular fibers, since $(\SL_2\ZZ)^{ab} \cong \ZZ/12\ZZ$.

\section{Index of the Braid Monodromy}\label{sec:index}

Let $\pi : M \to S$ be a Lefschetz fibration of genus $g$ with $\Delta$ its singular locus. As previously discussed, $\pi$ induces a subgroup $\Mod_\pi < \Mod(S,\Delta)$ of mapping classes which lift to a fiber-preserving diffeomorphism of $M$ (see \cref{defn:mod-sub-pi}). At this point, we are equipped to prove \Cref{thm:inf-index}, which we restate now in genus $g \geq 2$.
\begin{theorem}\label{prop:inf-index-sphere-auroux}
	Let $\pi : M \to S^2$ be a genus $g$ Lefschetz fibration over the sphere with singular locus $\Delta_\pi \subseteq S^2$. Then for the fiber-sum $\pi \fibsum \pi : M \fibsum M \to S^2$, $\Br(\pi \fibsum \pi)$ has infinite index in $\Mod(S^2, \Delta_{\pi \fibsum \pi})$.
\end{theorem}

\begin{proof}
	Let $\de_1,\ldots,\de_n$ be vanishing cycles corresponding to the monodromy factorization $T_{\de_1}\cdots T_{\de_n} = \Id$ of $\pi$. Furthermore let $\operatorname{Mon}(\pi) \coloneqq \langle T_{\de_1},\ldots,T_{\de_n} \rangle$. We apply Auroux's lemma, which shows that if $T_{\delta_1},\ldots,T_{\delta_n}$ is a factorization of a central element in $\Mod(\Sigma_g)$ then $T_{\delta_1},\ldots,T_{\delta_n}$ is Hurwitz equivalent to $T_{f \delta_1},\ldots,T_{f\delta_n}$ for any $f \in \operatorname{Mon}(\pi)$ \cite[Lemma 6(b)]{auroux-stable}. Since $T_{\de_1}\cdots T_{\de_n} = \Id$, there must be some pair $\de_i,\de_j$ so that $i(\de_i,\de_j) \geq 1$, as otherwise these Dehn twists would generate a free abelian group. Let $f_k = T_{\de_i}^k \in \Mon(\pi)$, then by Auroux's lemma there is a Hurwitz equivalence
	\begin{align*}
		T_{\delta_1},\ldots,T_{\delta_n},T_{\delta_1},\ldots,T_{\delta_n} \sim T_{\delta_1},\ldots,T_{\delta_n},T_{f_k\delta_1},\ldots,T_{f_k\delta_n}
	\end{align*}
	where $T_{\delta_1},\ldots,T_{\delta_n}T_{\delta_1},\ldots,T_{\delta_n}$ is the monodromy factorization corresponding to of $\pi \fibsum \pi$. Thus each of these factorizations lie in the same Hurwitz orbit. By \Cref{thm:moishezon-matsumuto} it suffices to show there are infinitely many up to simultaneous conjugacy. A conjugation by $h \in \Mod(\Sigma_g)$ will not change the intersection number of vanishing cycles, and hence two such factorizations are distinguished by $i(\de_j,T_{\de_i}^k \de_j) = \abs{k}i(\de_i,\de_j)^2$ \cite[Proposition 3.2]{primer}.
\end{proof}

Part (1) of \Cref{thm:inf-index}, restated here for convenience, follows as a corollary.
\begin{corollary}\label{cor:ell-inf-index}
	For any nontrivial genus one Lefschetz fibration $\pi : M \to S^2$, $\Br(\pi)$ has infinite index in $\Mod(S^2,\Delta)$, where $\Delta \subseteq S^2$ is the singular locus.
\end{corollary}

\begin{proof}
	Genus one Lefschetz fibrations are classified as iterated fiber sums of the standard elliptic fibration $E(1)$ with 12 nodal fibers\footnote{For a definition of $E(1)$, see \cite[\S 3.1]{gompf-stipsicz}} due to Moishezon (see \Cref{thm:moishezon}). Thus, it suffices to verify the result for $E(1)$. The monodromy of $E(1)$ generates $\Mod(T^2) \cong \SL_2\ZZ$ and for a geometric symplectic basis $\alpha,\beta$ of $T^2$ the monodromy factorization is $T_\alpha,T_\beta,\ldots,T_\alpha,T_\beta$ of length $12$. Note that Auroux's lemma applies for the subfactorization $(T_\alpha T_\beta)^3$ which acts as $-\Id$ on $H_1(T^2;\ZZ)$, and hence is central in $\Mod(T^2) \cong \Aut(H_1(T^2;\ZZ))$. The proof of \Cref{prop:inf-index-sphere-auroux} thus generalizes exactly to this setting.
\end{proof}

\begin{remark}
	The proof of \Cref{prop:inf-index-sphere-auroux} given above is inspired by a proof of Lee and Servan. They show that for any Lefschetz fibration $\pi$ of genus $g \geq 2$ over the sphere the fiber sum $\pi \fibsum \pi$ admits infinitely many homologically distinct sections \cite{seraphina-carlos}. 
\end{remark}

In \Cref{subsec:inf-index-general}, we will extend \Cref{prop:inf-index-sphere-auroux} to base surfaces with genus and/or boundary. Before doing so, we will first examine some exceptions where $\Br(\pi)$ is finite index over the disk.

\section{Finite-Index Exceptions}\label{subsec:finite-index}

Let $D$ be the unit disk with boundary, $D_n$ the disk with $n$ punctures, and fix some base point $b \in \partial D$. Furthermore let $\alpha,\beta$ be a geometric symplectic basis for the torus $T^2$. Consider the family $q_n : M_n \to D$ of Lefschetz fibrations given by the monodromy representations
\begin{align*}
	\phi_n : \pi_1(D_n, b) &\to \Mod(T^2) \cong \SL_2\ZZ \\
	\gamma_{2i+1} &\mapsto T_\alpha \cong \begin{pmatrix} 1 & -1 \\ 0 & 1 \end{pmatrix} \\
	\gamma_{2i} &\mapsto T_\beta \cong \begin{pmatrix} 1 & 0 \\ 1 & 1 \end{pmatrix},
\end{align*}
where $\gamma_i$ is a standard system of generators for $\pi_1(D_n,b)$, i.e., each $\gamma_i$ surrounds a single puncture counterclockwise and their product is homotopic to the boundary $\partial D_n$. A representation $\phi : \pi_1(D_n,b) \to \Mod(T^2)$ arising from a genus one fibration can be identified with a tuple $(\phi(\gamma_i))_{i=1}^n = (T_{\de_1},\ldots,T_{\de_n})$ of Dehn twists. The Hurwitz action given by the braid group $B_n = \Mod(D_n,\partial D_n)$ on the representation $\phi : \pi_1(D_n,b) \to \Mod(T^2)$ acts as follows for the standard half-twist generators $\sigma_i$ of $B_n$:
\begin{align}
	\sigma_i : (T_{\de_1},\ldots,T_{\de_n}) \mapsto (T_{\de_1},\ldots, T_{\de_i}T_{\de_{i+1}}T_{\de_{i}}^{-1}, T_{\de_i}, \ldots,T_{\de_n}) = (T_{\de_1},\ldots,T_{T_{\de_i} \de_{i+1}}, T_{\de_i},\ldots,T_{\de_n}). \label{eq:hurwitz}
\end{align}
Given the setup above, we show \Cref{thm:index-ell-disk-intro}, restated here for convenience.
\begin{theorem}\label{prop:index-for-ell-disk}
	For $n \leq 4$, $\Br(q_n)$ is finite index in $B_n$, with indices
	\begin{align*}
		[B_2 : \Br(q_2)] &= 3 && [B_3 : \Br({q_3})] = 8  && [B_4 : \Br({q_4})] = 27.
	\end{align*}
	For $n \geq 5$, $\Br({q_n})$ has infinite index in $B_n$.
\end{theorem}
We delay the proof of the second statement until \Cref{subsec:inf-index-general}, where we will prove a more general result.

\begin{proof}[Proof when $n \leq 4$]
	For $n \leq 4$, it suffices to compute the Hurwitz orbit. One can do so explicitly by constructing a directed graph whose edges are generators and whose vertices are representations of $\pi_1(D_n) \cong F_n$. The orbit is complete provided that the in/out-degree of each vertex is $n-1$, the number of generators of $B_n$. This computation is performed in Sage/Python, the code for which can be found at \eqref{eq:code}. These orbits are displayed in \Cref{fig:braid-orbits-intro} for $n=3,4$, the orbit for $n=2$ is simply a triangle.
\end{proof}

We will show in \Cref{subsec:inf-index-general} that $[B_n : \Br({q_n})]$ is infinite for $n \geq 5$. In this section we dedicate ourselves to the analysis of these finite-index exceptions where $n \leq 4$. Let $\sigma_{ij}$ be the braid exchanging puncture $i$ and $j$ along the interiors of $\gamma_i$ and $\gamma_j$, as usual we abbreviate $\sigma_i = \sigma_{i(i+1)}$. It is well-known that $\sigma_{ij} \in \Br({q_n})$ when $i \equiv j \pmod{2}$, and that $\sigma_{ij}^3 \in \Br({q_n})$ when $i \not\equiv j \pmod{2}$ (see \cite[Remark 5.4]{fl-smooth-mw} and \cite[Main Theorem]{lonne-braid-monodromy}). Using the orbit graphs in \Cref{fig:braid-orbits-intro}, one can find representatives for the cosets of $\Br({q_n})$ in $B_n$ for $n \leq 4$, and subsequently compute a presentation for $\Br({q_n})$ in GAP for $n \leq 4$ \cite{GAP4, FR2.4.13}.

\begin{proposition}\label{prop:fin-index-gen-set}
	$\Br({q_3})$ and $\Br({q_4})$ are generated as follows:
	\begin{align}
		\Br({q_3}) &= \langle \sigma_1^3, \sigma_{13} \rangle && \Br({q_4}) = \langle \sigma_1^3, \sigma_{13}, \sigma_{24} \rangle.
	\label{eq:presentations}
	\end{align}
	As a consequence, one finds that the abelianizations $\Br({q_3})^{ab}, \Br({q_4})^{ab}$ are free of rank $2$ and that $[PB_n : PB_n \cap \Br({q_n})] = [B_n : \Br({q_n})]$ for $n \leq 4$ (here $PB_n$ is the pure braid group). 	\begin{align*}
		\Br({q_3}) &= \langle \sigma_1^3,\sigma_{13} \mid (\sigma_{13}\sigma_1^3)^3 = (\sigma_1^3\sigma_{13})^3 \rangle.
	\end{align*}
	The analogous presentation of $\Br({q_4})$ is quite complex, requiring 13 relations where the maximum length relation has word length $36$.
\end{proposition}
\Cref{prop:fin-index-gen-set-intro} from the introduction is simply the first part of \Cref{prop:fin-index-gen-set}. Notably, these generating sets agree with the generating sets derived in the algebraic category by L\"{o}nne. Having analyzed the finite-index exceptions, we turn our attention back to the general case.

\section{General Criteria}\label{subsec:inf-index-general}

We'll now devote ourselves to extending \Cref{prop:inf-index-sphere-auroux} over arbitrary compact connected oriented base surfaces $B$. We develop multiple criteria guaranteeing that $\Mod_\pi = \im(\Mod(\pi) \to \Mod(B,\Delta))$ has infinite index for a Lefschetz fibration $\pi : M \to B$ with singular fibers along $\Delta \subseteq B \setminus \partial B$. To do so succinctly, we require some additional language.
\begin{defn}[Moishezon spider]\label{defn:spider}
	Let $\pi : M \to B$ be a Lefschetz fibration of genus $g$ over a compact connected oriented base $B$ with singular locus $\Delta_\pi$. A \emph{Moishezon spider} for $\pi$ consists of the following data
	\begin{enumerate}
		\item An embedded disk $D$ in $B$, so that $\partial D \cap \Delta_\pi = \0$.
		\item A basepoint $b \in \partial D$, where if $B$ has boundary we require $b \in \partial D \cap \partial B$.
		\item A collection of arcs $a_1,\ldots,a_k$ in $D$ beginning at $b$ and ending at distinct singular values $p_1,\ldots,p_k \in \operatorname{int}(D) \cap \Delta_\pi$ of $\pi$, disjoint from $\partial D$ except at $b$. The arcs $a_1,\ldots,a_k$ are ordered counterclockwise by how they meet a sufficiently small neighborhood of $b$.
	\end{enumerate}
	The \textit{topological type} of a Moishezon spider is the collection of vanishing cycles $(\de_1,\ldots,\de_k)$ associated to $a_1,\ldots,a_k$ via $\pi$ in the fiber $\Sigma_g \cong \pi^{-1}(b)$ and is well-defined up to the action of $\Mod(\Sigma_g)$. The arcs $a_1,\ldots,a_k$ are called the \textit{legs of the spider}.
\end{defn}
Our first result concerns spiders of type $(\de_1,\de_2)$ where the geometric intersection number $i(\de_1,\de_2) \geq 2$.
\begin{lemma}\label{lemma:inf-int-2-boundary}
	Suppose that a genus $g$ Lefschetz fibration $\pi : M \to B$ over a compact connected surface $B$ with one boundary component admits a Moishezon spider with topological type $(\de_1,\de_2)$, where $i(\de_1,\de_2) \geq 2$. Then $\Mod_\pi$ has infinite index in $\Mod(B,\Delta_\pi)$.
\end{lemma}

\begin{proof}
	Let $b \in \partial B$ and take $\gamma_1,\gamma_2 \in \pi_1(B \setminus \Delta_\pi,b)$ simple loops about singular fibers $p_1,p_2$ of $\pi$ with vanishing cycles $\de_1,\de_2$ respectively. The monodromy factorization $W$ of $\pi$ thus contains $W' = T_{\de_1}T_{\de_2}$ as a subword. Let $\sigma \in \Mod(B,\Delta_\pi)$ be the braid which exchanges $p_1,p_2$ and is supported on the component of $B \setminus \gamma_1 \cup \gamma_2$ containing $p_1,p_2$. We show that $\sigma^k \cdot W$ are distinct for all $k$, and hence the Hurwitz orbit is infinite, proving the lemma. To do so, note that $\sigma$ acts only on the subword $W'$. That is $\sigma$ acts on the representation
	\begin{align*}
		\phi' = \rest{\phi}{\langle \gamma_1,\gamma_2 \rangle}  : \langle \gamma_1,\gamma_2 \rangle \to \langle T_{\de_1},T_{\de_2} \rangle,
	\end{align*}
	where $\phi : \pi_1(B \setminus \Delta_\pi,b) \to \Mod(\Sigma_g)$ is the monodromy representation of $\pi$. By assumption, $i(T_{\de_1},T_{\de_2}) \geq 2$ and so $\langle T_{\de_1},T_{\de_2} \rangle \cong F_2$, the free group on two elements \cite[Theorem 3.14]{primer}. Let $\eta : F_2 \to F_2$ be the automorphism
	\begin{align*}
		\eta(T_{\de_1}) = T_{\de_2}, \psi(T_{\de_2}) = T_{\de_2}^{-1}T_{\de_1}T_{\de_2}.
	\end{align*}
	We find that by definition of the Hurwitz action $ \sigma \cdot \phi' = \eta \circ \phi'$. Hence $\sigma^k \cdot \phi = \sigma^j \phi$ if and only if $\eta^k \circ \phi' = \eta^j \circ \phi'$. Since $\phi'$ is surjective, this implies that $\eta^{k-j} = \Id$. But $\eta$ has infinite order, since the Artin representation $B_2 \to \Aut(F_2)$ mapping $\sigma$ to $\eta$ is faithful \cite[Theorem 14]{artin-braids}.
\end{proof}
The proof generalizes directly to the following lemma.
\begin{lemma}\label{lemma:inf-int-2}
	Let $\pi : M \to B$ be a Lefschetz fibration over a compact connected, oriented base $B$ with at most one boundary component containing a Moishezon spider  of type $(\de_1,\de_2,\de_1,\de_2)$ over singular values $p_1,p_2,p_3,p_4$ such that $i(\de_1,\de_2) \geq 2$. Then $\Mod_{\pi}$ has infinite index.
\end{lemma}

\begin{proof}
	It suffices to show the Hurwitz orbit of the monodromy factorization is infinite up to global conjugation. We apply the proof of \Cref{lemma:inf-int-2-boundary} where $\sigma$ exchanges $p_1,p_2$ and note that if $\sigma^k \cdot \phi$ is conjugate to $\sigma^j \cdot \phi$ by $f \in \Mod(\Sigma_g)$, then $f(\de_1) = \de_1, f(\de_2) = \de_2$. Hence $f$ commutes with $\langle T_{\de_1},T_{\de_2} \rangle$, and so $\sigma^k \cdot \phi = f\sigma^j \phi f^{-1}$ if and only if $\eta^k \circ \phi' = \eta^j \circ \phi'$ as before.
\end{proof}

\begin{lemma}\label{lemma:int-1-to-int-2}
	Suppose a genus $g$ Lefschetz fibration $\pi : M \to B$ contains a Moishezon spider of type $(\alpha,\beta, \alpha, \beta, \alpha)$ over singular values $p_1,\ldots,p_5$ where $i(\alpha,\beta) = 1$. Then $\pi$ contains a Moishezon spider of type $(T_\alpha \beta, T_\beta \alpha)$.
\end{lemma}

\begin{proof}
	Let $\gamma_1,\ldots,\gamma_5$ be simple loops from $b \in B$ (if $B$ has boundary, take $b \in \partial B$) surrounding $p_1,\ldots,p_5$ respectively so that the monodromy about $\gamma_{2i+1}$ is $T_{\alpha}$ and the monodromy about $\gamma_{2i}$ is $T_{\beta}$. Take a small disk containing $\gamma_1,\ldots,\gamma_5$ and $p_1,\ldots,p_5$ embedded in $B$. Then the braid group $B_5$ acts on the monodromy restricted to $\langle \gamma_1,\ldots,\gamma_5 \rangle$ (equivalently, the subword of the monodromy factorization).

	We then perform a local calculation over this disk and the torus-with-boundary  $\Sigma_1^1$ formed from a neighborhood of $\alpha,\beta$ in the fiber $\Sigma_g$ over $b$. The allowed moves are to apply the braid $\sigma_i$ to the vanishing cycles $\de_1,\de_2,\de_3,\de_4,\de_5$ and obtain
	\begin{align*}
		\sigma_i : \de_1,\ldots,\de_i,\de_{i+1},\ldots,\de_5 \mapsto \de_1,\ldots,T_{\de_i}\de_{i+1},\de_i,\ldots,\de_5,
	\end{align*}
	and the inverse. We apply the word $\sigma_2\sigma_1^2\sigma_4$
	\begin{align*}
		(\alpha,\beta,\alpha,\beta,\alpha) &\xmapsto{\sigma_4} (\alpha,\beta,\alpha,T_\beta \alpha,\beta) \\
										   &\xmapsto{\sigma_1^2} (T_\alpha T_\beta T_{\alpha}^{-1} \alpha, T_\alpha \beta, \alpha, T_\beta \alpha, \beta) = (\beta,T_\alpha \beta, \alpha, T_\beta \alpha, \beta) \\
										   &\xmapsto{\sigma_2} (\beta, \beta, T_\alpha \beta, T_\beta \alpha, \beta).
	\end{align*}
	Standard computations then show that $i(T_\alpha \beta, T_\beta \alpha) = 2$. See \Cref{fig:local-calc-torus} for a purely visual computation of $\sigma_2\sigma_1^2\sigma_4$ applied to the monodromy factorization---note that, in \Cref{fig:local-calc-torus}, the curves are labeled $(\de_1,\de_2,\de_3,\de_4,\de_5)$ in each step to keep track of their indexing. Alternatively, one can identify $\Mod(\Sigma_1^1)$ with $\SL_2\ZZ$ to perform the calculation\footnotemark. This completes the proof.
	\begin{figure}[h]
		\centering
		\includegraphics[width=0.75\textwidth]{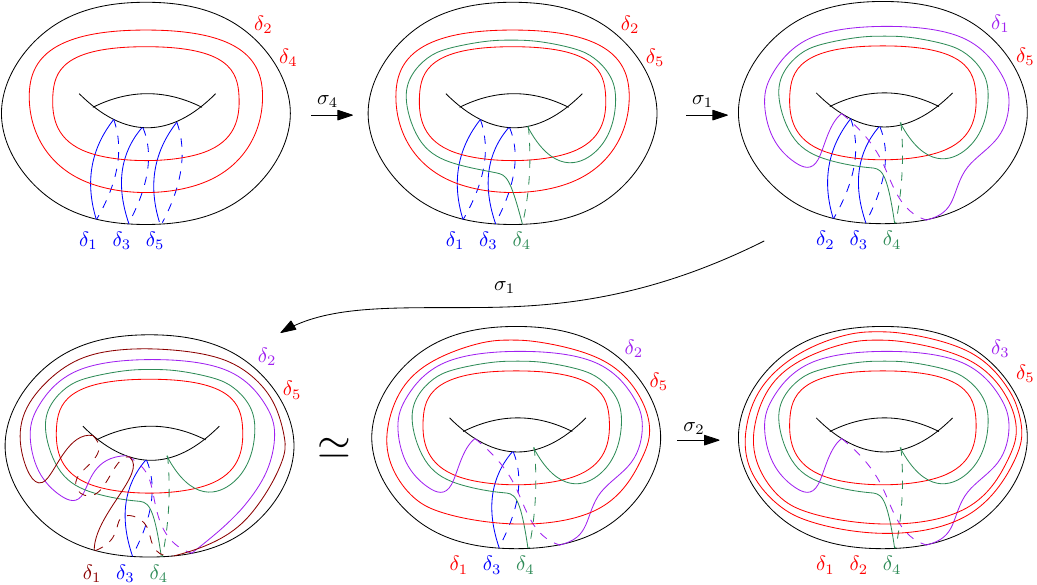}
		\caption{The Local Calculation of \Cref{lemma:int-1-to-int-2}.}
		\label{fig:local-calc-torus}
	\end{figure}
\end{proof}
\footnotetext{As one can guess, this is much easier in practice than directly computing with curves. In particular, one does not make so many orientation mistakes performing Dehn twists.}
\begin{remark}
	Direct calculation shows that $i(T_\alpha \beta, T_\beta \alpha) = 2$. We often use \Cref{lemma:int-1-to-int-2} to construct Moishezon spiders to apply \Cref{lemma:inf-int-2,lemma:inf-int-2-boundary}.
\end{remark}

Combining \Cref{lemma:inf-int-2-boundary} and \Cref{lemma:int-1-to-int-2} yields the following corollary promised in \Cref{subsec:finite-index}.
\begin{corollary}
	Let $n \geq 5$ and $q_n : M_n \to D^2$ be the genus one Lefschetz fibration with monodromy
	\begin{align*}
		\phi_n : \pi_1(D^2 \setminus \{p_1,\ldots,p_n\}) \to \SL_2\ZZ \\
		\phi_n(\gamma_{2i+1}) = \begin{pmatrix} 1 & -1 \\ 0 & 1 \end{pmatrix}, \quad \phi_n(\gamma_{2i}) = \begin{pmatrix} 1 & 0 \\ 1 & 1 \end{pmatrix}.
	\end{align*}
	Then
	\[
		[B_n : \Br(q_n)] = \infty
	\]
\end{corollary}

\Cref{lemma:int-1-to-int-2,lemma:inf-int-2} also yield a new proof of Part (1) of \Cref{thm:inf-index}.
\begin{proof}[Alternate proof of Part (1) of \Cref{thm:inf-index}]
	Via Moishezon's Theorem \cite{moishezon}, any genus one Lefschetz fibration contains a Moishezon spider of type
	\begin{align*}
		(\alpha,\beta, \ldots, \alpha, \beta)
	\end{align*}
	with $12$ legs, such that $i(\alpha,\beta) = 1$. One then applies \Cref{lemma:int-1-to-int-2} twice before applying \Cref{lemma:inf-int-2}.
\end{proof}

Finally, we obtain a mild generalization of Part (2) of \Cref{thm:inf-index} to arbitrary base.
\begin{corollary}\label{cor:inf-index-general}
	Let $\pi : M \to B$ be a Lefschetz fibration of genus $g$, where $B$ is any connected, oriented surface, which admits a Moishezon spider of topological type $(\alpha,\beta)$ so that $i(\alpha,\beta) \neq 0$, then $\Mod_{\pi \fibsum \pi \fibsum \pi \fibsum \pi}$ has infinite index in $\Mod(B \connsum B \connsum B \connsum B)$.
\end{corollary}
The proof is identical to the proof above of Part (1) of \Cref{thm:inf-index}. Three of the fiber sums are used to produce the five vanishing cycles of \Cref{lemma:int-1-to-int-2}. The remaining two vanishing cycles serve as a marking.

\section{Proof of \texorpdfstring{\Cref{thm:holomorphic-genus-two}}{Theorem 1.3}}\label{sec:holomorphic}

Let $\Sigma_g$ be a connected oriented surface of genus $g \geq 2$. Fix curves $\alpha_1,\ldots,\alpha_{2g+1}$ on $\Sigma_g$ as in \Cref{fig:hyperelliptic}. With $a_i = T_{\alpha_i}$, the following are monodromy factorizations in $\Mod(\Sigma_g)$
\begin{align}
	(a_1\cdots a_{2g+1}^2 \cdots a_2a_1)^2 &= 1 \label{eq:eta-1} \\
	(a_1\cdots a_{2g})^{2(2g+1)} &= 1 \label{eq:eta-2} \\
	(a_1\cdots a_{2g+1})^{2g+2} &= 1 \label{eq:eta-3},
\end{align}
see \cite[Proposition 4.12, \S 5.1.4]{primer}. Each of these monodromy factorizations gives rise to a Lefschetz fibration:
\begin{align}
	\eta_1 : X(1) \to S^2 \\
	\eta_2 : X(2) \to S^2 \\
	\eta_3 : X(3) \to S^2,
\end{align}
where $X(1),X(2),X(3)$ are the total spaces of these fibrations. 
\begin{figure}
	\centering
	\includegraphics{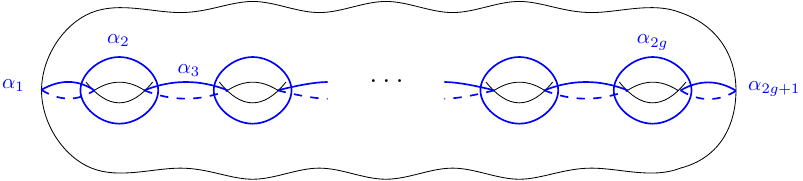}
	\caption{Curves $\alpha_1,\ldots,\alpha_{2g+1}$ on $\Sigma_g$}
	\label{fig:hyperelliptic}
\end{figure}
\Cref{thm:holomorphic-genus-two} follows from the following theorem.

\begin{theorem}\label{thm:three-examples}
	For any genus $g \geq 2$, each of the three Lefschetz fibrations $\eta = \eta_1,\eta_2,\eta_3$ satisfies 
	\[
		[\Mod(S^2,\Delta_\eta) : \Br(\eta)] = \infty.
	\]
\end{theorem}
We first sketch the implication of \Cref{thm:holomorphic-genus-two} from \Cref{thm:three-examples}.

\begin{proof}[\Cref{thm:three-examples} implies \Cref{thm:holomorphic-genus-two}]
	Let $\pi : M \to S^2$ be a nontrivial genus $g = 2$ holomorphic Lefschetz fibration whose vanishing cycles are all nonseparating. By Chakiris' classification, $\pi$ must be a fiber sum $\eta_1^{\fibsum k_1} \fibsum \eta_2^{\fibsum k_2} \fibsum \eta_3^{\fibsum k_1}$, where $\eta_i^{\fibsum k}$ denotes the fiber sum of $\eta_i$ with itself $k$ times \cite[Main Theorem]{chakiris}. Let $\rho_\pi : \pi_1(S^2 \setminus \Delta_\pi) \to \Mod(\Sigma_g)$ be the monodromy representation of $\pi$. The Hurwitz orbit of $[\rho_\pi]$ contains the Hurwitz orbits for each component of the fiber sum, at least one of which is infinite by \Cref{thm:three-examples} since $\pi$ is nontrivial. Hence the Hurwitz orbit of $[\rho_\pi]$ is infinite. The result follows as $\Br(\pi) = \Stab [\rho_\pi]$ by \Cref{thm:moishezon-matsumuto}.
\end{proof}
With the implication above in hand, we conclude by proving \Cref{thm:three-examples}.

\begin{proof}[Proof of \Cref{thm:three-examples}]
	Let 
	\[
		\rho : \pi_1(S^2 \setminus \Delta_{\eta_1}) \to \Mod(\Sigma_g)
	\]
	be the monodromy representation of the Lefschetz fibration $\eta_1$. It suffices to show that the conjugacy class $[\rho]$ has infinite orbit under the Hurwitz action of $\Mod(S^2,\Delta_{\eta_1})$ by Matsumuto's theorem (see \Cref{thm:moishezon-matsumuto} above). By inspection of the monodromy factorization \eqref{eq:eta-1}, $\eta_1$ contains a Moishezon spider of type
	\begin{align}
		(\de_1,\ldots,\de_8) = (\alpha_1,\alpha_2,\alpha_2,\alpha_1,\alpha_1,\alpha_2,\alpha_2,\alpha_1),\label{spider:basic}
	\end{align}
	where $i(\alpha_1,\alpha_2) = 1$. By applying the argument of \Cref{lemma:int-1-to-int-2}, we produce from the spider \eqref{spider:basic} a new Moishezon spider of type
	\begin{align*}
		(\de_1',\ldots,\de_8') = (\alpha_2,\de_2', \alpha_2, T_{\alpha_1}\alpha_2, \alpha_1, T_{\alpha_2}\alpha_1, \alpha_2, \alpha_2),
	\end{align*}
	where $i(\alpha_2,\de_2') = 1$. Let $\gamma_1,\ldots,\gamma_8$ be simple loops about single punctures with monodromies $\rho(\gamma_i) = T_{\de_i'}$, the Dehn twist about $\de_i'$. Let $\sigma$ be the braid exchanging the punctures corresponding to $\gamma_4,\gamma_6$. It suffices to show that if $\sigma^k \cdot \rho$ is conjugate to $\sigma^j \cdot \rho$ then $k = j$. Suppose there exists some $f \in \Mod(\Sigma_g)$ so that 
	\begin{equation}
		f(\sigma^k \cdot \rho)f^{-1} = \sigma^j \cdot \rho.\label{eq:conj-equal}
	\end{equation}
	Evaluating both sides of \eqref{eq:conj-equal} on $\gamma_1,\gamma_2$ yields
	\begin{align*}
		fT_{\alpha_2}f^{-1} = f\rho(\gamma_1)f^{-1} = \rho(\gamma_1) = T_{\alpha_2} && fT_{\de_2'}f^{-1} = f\rho(\gamma_2)f^{-1} = \rho(\gamma_2) = T_{\de_2'},
	\end{align*}
	since $\sigma$ is supported away from $\gamma_1,\gamma_2$. Note that $\langle T_{\alpha_1},T_{\alpha_2} \rangle = \langle T_{\alpha_1},T_{\de_2'}\rangle$ since these both generate the mapping class group of the torus-with-boundary formed from a tubular neighborhood of $\alpha_1,\alpha_2$. Now consider the restrictions
	\begin{align*}
		\rest{f(\sigma^k \cdot \rho)f^{-1}}{\langle \gamma_4, \gamma_6 \rangle } = \rest{(\sigma^j \cdot \rho)}{\langle \gamma_4,\gamma_6 \rangle } : \langle \gamma_4,\gamma_6 \rangle \to \langle T_{T_{\alpha_1} \alpha_2}, T_{T_{\alpha_2} \alpha_1} \rangle \cong F_2,
	\end{align*}
	where the final isomorphism holds because $i(T_{\alpha_1} \alpha_2, T_{\alpha_2} \alpha_1) \geq 2$ \cite[Theorem 3.14]{primer}. Applying the equality $T_{T_{\alpha_i} \alpha_j} = T_{\alpha_i}T_{\alpha_j}T_{\alpha_i}^{-1}$ and the fact that $f$ acts by the identity via conjugation on $\langle T_{\alpha_1},T_{\alpha_2} \rangle$ yields that
	\begin{align*}
		\sigma^k \cdot \rest{\rho}{\langle \gamma_4,\gamma_6 \rangle} = \sigma^j \cdot \rest{\rho}{\langle \gamma_4,\gamma_6 \rangle} : \langle \gamma_4,\gamma_6 \rangle \to F_2.
	\end{align*}
	As in the proof of \Cref{lemma:inf-int-2-boundary}, the equality $\sigma^k \cdot \rho = \sigma^j \cdot \rho$ must then imply that $k = j$, by the faithfulness of the Artin representation $B_2 \to \Aut(F_2)$ \cite[Theorem 14]{artin-braids}. Thus $[\rho]$ has infinite Hurwitz orbit, and $[\Mod(S^2,\Delta_{\eta_1}) : \Br(\eta_1)] = \infty$.

	We now proceed to $\eta_2$ and $\eta_3$ respectively. By inspection of the monodromy factorization \eqref{eq:eta-2} (resp. \eqref{eq:eta-3}), we see that $\eta_2$ (resp. $\eta_3$) contains a Moishezon spider of type
	\begin{align*}
		(\de_1,\ldots,\de_{10}) = (\alpha_1,\alpha_2,\alpha_1,\alpha_2,\ldots,\alpha_1,\alpha_2).
	\end{align*}
	Applying \Cref{lemma:int-1-to-int-2} twice yields a Moishezon spider of type 
	\[
		(T_{\alpha_1}\alpha_2, T_{\alpha_2}\alpha_1, T_{\alpha_1}\alpha_2, T_{\alpha_2}\alpha_1).
	\]
	Therefore, \Cref{lemma:inf-int-2} implies that $\Br(\eta_2)$ (resp. $\Br(\eta_3)$) has infinite index inside of $\Mod(S^2,\Delta_{\eta_2})$ (resp. $\Mod(S^2,\Delta_{\eta_3})$).
\end{proof}

\printbibliography

\end{document}

%% file: research-notes-preamble.tex
\usepackage{subfiles}

\usepackage{amsfonts,amsthm,amssymb,amsmath,amscd,euscript}
\usepackage{nccmath}
\usepackage{nicematrix}
\usepackage{mathdots}
\usepackage{csquotes}
\usepackage{mathrsfs}
\usepackage{wasysym}
\usepackage{bbm}
\usepackage{ifthen}
\usepackage{xpatch}
\usepackage{thmtools}
\usepackage{thm-restate}
\usepackage{calc}
\usepackage[reftex]{theoremref}

\usepackage[backend=biber,style=alphabetic,giveninits=true]{biblatex}
\usepackage[english]{babel}
\usepackage[hypertexnames=false]{hyperref}

\usepackage{multicol}
\usepackage[dvipsnames,rgb]{xcolor}
\usepackage{framed}
\usepackage{pdfpages}
\usepackage{scalerel}
\usepackage{fullpage}
\usepackage{color}
\usepackage{lipsum}
\usepackage{bold-extra}
\usepackage{fancyhdr}

\usepackage{tikz,tikz-3dplot}
\usepackage{tkz-berge}

\usepackage{pgfplots}
\pgfplotsset{compat=1.7}
\usetikzlibrary{shapes, arrows}
\usetikzlibrary{shapes.geometric}
\usetikzlibrary{graphs}
\usetikzlibrary{graphs.standard}
\usetikzlibrary{decorations.pathmorphing}
\usetikzlibrary{decorations.markings}
\tikzset{->-/.style={decoration={
  markings,
  mark=at position #1 with {\arrow{>}}},postaction={decorate}},
  ->-/.default=0.53
}
\tikzset{-<-/.style={decoration={
  markings,
  mark=at position #1 with {\arrow{<}}},postaction={decorate}},
  -<-/.default=0.53
}
\pgfplotsset{soldot/.style={color=blue,only marks,mark=*}} \pgfplotsset{holdot/.style={color=blue,fill=white,only marks,mark=*}}
\usepackage{tikz-cd}
\usepackage[all,2cell,cmtip]{xy}
\UseTwocells

\usepackage{mathtools}
\usepackage{float}
\usepackage[margin=1in,heightrounded]{geometry}
\usepackage{etoolbox}
\usepackage[most]{tcolorbox}
\tcbuselibrary{theorems}
\tcbuselibrary{skins}
\usepackage{mdframed}
\usepackage[inline, shortlabels]{enumitem}
\SetEnumitemKey{twocol}{itemsep = 1\itemsep,parsep  = 1\parsep,before  = \raggedcolumns\begin{multicols}{2},after   = \end{multicols}}
\usepackage{tasks}
\usepackage{stackengine}
\usepackage{tabularx}
\usepackage{soul}
\usepackage[a]{esvect}
\usepackage[activate={true,nocompatibility},nopatch=footnote,final,tracking=true,kerning=true,spacing=true,factor=1100,stretch=10,shrink=10]{microtype}
\microtypecontext{spacing=nonfrench}
\usepackage{siunitx}
\usepackage{mathrsfs}
\usepackage{quiver} 
\usepackage{diagbox}
\DeclareFontFamily{U}{mathx}{\hyphenchar\font45}
\DeclareFontShape{U}{mathx}{m}{n}{
      <5> <6> <7> <8> <9> <10>
      <10.95> <12> <14.4> <17.28> <20.74> <24.88>
      mathx10
      }{}
\DeclareSymbolFont{mathx}{U}{mathx}{m}{n}
\DeclareFontSubstitution{U}{mathx}{m}{n}
\DeclareMathAccent{\widecheck}{0}{mathx}{"71}
\DeclareMathAccent{\wideparen}{0}{mathx}{"75}

\usepackage[color=red!60, colorinlistoftodos]{todonotes}
\usepackage{setspace}
\usepackage{stmaryrd}
\usepackage{centernot}
\usepackage{subfig}
\usepackage{cancel}
\newcommand\Ccancel[2][black]{
    \let\OldcancelColor\CancelColor
    \renewcommand\CancelColor{\color{#1}}
    \cancel{#2}
    \renewcommand\CancelColor{\OldcancelColor}
}
\usepackage{makecell}
\hypersetup{colorlinks=true,citecolor=blue,urlcolor =blue,linkbordercolor={1 0 0}}

\definecolor{bluepigment}{rgb}{0.2, 0.2, 0.6}
\definecolor{cobalt}{rgb}{0.0, 0.28, 0.67}
\definecolor{darkpowderblue}{rgb}{0.0, 0.2, 0.6}
\definecolor{egyptianblue}{rgb}{0.06, 0.2, 0.65}
\definecolor{lapislazuli}{rgb}{0.15, 0.38, 0.61}
\definecolor{tealblue}{rgb}{0.21, 0.46, 0.53}
\definecolor{darktealblue}{rgb}{0.11, 0.36, 0.43}
\definecolor{zaffre}{rgb}{0.0, 0.08, 0.66}
\definecolor{zaffreprime}{rgb}{0.4, 0.18, 0.76}
\definecolor{arsenic}{rgb}{0.23, 0.27, 0.29}
\definecolor{bulgarianrose}{rgb}{0.28, 0.02, 0.03}
\definecolor{aurometalsaurus}{rgb}{0.43, 0.5, 0.5}
\definecolor{amethyst}{rgb}{0.6, 0.4, 0.8}

\definecolor{pblue}{rgb}{.114, .102, .365}
\definecolor{pblue2}{rgb}{.204, .196, .471}
\definecolor{pblue3}{rgb}{.655,.784,.855}
\definecolor{pgreen}{rgb}{.075,.212,.337}
\definecolor{pgreen3}{rgb}{.667,.882,.765}

\definecolor{codegreen}{rgb}{0,0.6,0}
\definecolor{codegray}{rgb}{0.5,0.5,0.5}
\definecolor{codepurple}{rgb}{0.58,0,0.82}
\definecolor{backcolour}{rgb}{0.95,0.95,0.92}

\allowdisplaybreaks[1]

\makeatletter
\patchcmd{\@maketitle}%
  {\ifx\@empty\@dedicatory}%
  {\ifx\@empty\@date \else {\vskip3ex \centering\footnotesize\@date\par\vskip1ex}\fi%
   \ifx\@empty\@dedicatory}%
  {}{}%
\patchcmd{\@adminfootnotes}%
  {\ifx\@empty\@date\else \@footnotetext{\@setdate}\fi}%
  {}{}{}%

\def\section{\@startsection{section}{1}%
  \z@{.7\linespacing\@plus\linespacing}{.5\linespacing}%
  {\normalfont\Large\bfseries}}
\def\subsection{\@startsection{subsection}{2}%
  \z@{.5\linespacing\@plus.7\linespacing}{.5\linespacing}%
  {\normalfont\large\bfseries}}%
\def\subsubsection{\@startsection{subsubsection}{3}%
  \z@{.5\linespacing\@plus.7\linespacing}{.5\linespacing}%
  {\normalfont\bfseries}}%
\def\@tocline#1#2#3#4#5#6#7{\relax%
  \ifnum #1>\c@tocdepth 
  \else%
    \par \addpenalty\@secpenalty\addvspace{#2}%
    \begingroup \hyphenpenalty\@M%
    \@ifempty{#4}{%
      \@tempdima\csname r@tocindent\number#1\endcsname\relax%
    }{%
      \@tempdima#4\relax%
    }%
    \parindent\z@ \leftskip#3\relax \advance\leftskip\@tempdima\relax%
    \rightskip\@pnumwidth plus4em \parfillskip-\@pnumwidth%
    #5\leavevmode\hskip-\@tempdima%
      \ifcase #1%
       \or\or \hskip 1em \or \hskip 2em \else \hskip 3em \fi%
      #6\nobreak\relax%
    \dotfill\hbox to\@pnumwidth{\@tocpagenum{#7}}\par%
    \nobreak%
    \endgroup%
  \fi}%
\makeatother%

\makeatletter%
\xpatchcmd{\endmdframed}%
    {\aftergroup\endmdf@trivlist\color@endgroup}%
    {\endmdf@trivlist\color@endgroup\@doendpe}%
    {}{}%
\makeatother%

\newenvironment{solution} {\begin{mdframed}[skipabove=0.1in,skipbelow=0.1in]\begin{proof}[Solution]} {\end{proof}\end{mdframed}}

\mdfdefinestyle{theoremstyle}{fontcolor=amethyst, linecolor=amethyst, bottomline=false,topline=false,rightline=false,skipabove=2pt,skipbelow=2pt}

\declaretheorem[numberwithin=section, name=Theorem]{theorem}

\declaretheorem[ name=Lemma, sibling=theorem]{lemma}

\declaretheorem[ name=Corollary, sibling=theorem]{corollary}

\declaretheorem[sibling=theorem, name=Proposition]{proposition}
\declaretheorem[numberwithin=section, style=definition, name=Definition]{defn}

\declaretheorem[,numberwithin=section, name=Remark,sibling=theorem,style=definition]{remark}
\declaretheorem[sibling=theorem,numberwithin=section, name=Remarks,style=definition]{remarks}

\declaretheorem[sibling={example}, name=Exercise]{exercise}

\declaretheorem[sibling={example}, name=Problem]{prob}
\declaretheorem[sibling={example}, name=Question]{quest}

\declaretheorem[numbered=no, name=Problem]{prob-nonum}

\lstdefinestyle{mystyle}{
    keywordstyle=\color[rgb]{0,0,1},
    commentstyle=\color[rgb]{0.026,0.412,0.595},
    stringstyle=\color[rgb]{0.627,0.126,0.941},
    numberstyle=\color[rgb]{0.205, 0.142, 0.73},
    frame=single,
    breakatwhitespace=false,
    breaklines=true,
    postbreak=\mbox{\textcolor{darkgray}{$\hookrightarrow$}\space},
    captionpos=b,
    keepspaces=true,
    numbers=left,
    numbersep=5pt,
    showspaces=false,
    showstringspaces=false,
    showtabs=false,
    tabsize=2,
    basicstyle=\ttfamily,
    columns=fullflexible,
    keepspaces=true,
    language=C++
}

\DeclareRobustCommand{\SkipTocEntry}[3]{}

\usepackage{cleveref}

\crefname{defn}{definition}{definitions}
\Crefname{defn}{Definition}{Definitions}
\crefname{theorem}{theorem}{theorems}
\Crefname{theorem}{Theorem}{Theorems}
\crefname{prob}{problem}{problems}
\Crefname{prob}{Problem}{Problems}
\crefname{exercise}{exercise}{exercises}
\Crefname{exercise}{Exercise}{Exercises}
\crefname{proposition}{proposition}{propositions}
\Crefname{proposition}{Proposition}{Propositions}
\makeatletter
\newcommand{\leqnomode}{\tagsleft@true\let\veqno\@@leqno}
\makeatother

\newcommand{\ZZ}{\mathbb{Z}}

\newcommand{\CC}{\mathbb{C}}

\newcommand{\abs}[1]{\left\lvert #1 \right\rvert}

\newcommand{\0}{\emptyset}

\renewcommand*{\st}{\;\vert\;}

\DeclareMathOperator{\Id}{Id}

\DeclareMathOperator{\im}{im}

\DeclareMathOperator{\Hom}{Hom}

\newcommand{\rest}[2]{{
  \left.\kern-\nulldelimiterspace 
  #1 
  \vphantom{\big|} 
  \right|_{#2} 
  }}

\DeclareMathOperator{\SL}{SL}

\newcommand{\de}{\delta}

\DeclareMathOperator{\Mod}{Mod}

\DeclareMathOperator{\Stab}{Stab}

\DeclareMathOperator{\Diff}{Diff}

\newcommand*{\PP}{\mathbb{P}}

\newcommand{\Bl}{\mathcal{B}\ell}
\DeclareMathOperator{\Aut}{Aut}
\DeclareMathOperator{\Mon}{Mon}

%% file: intro.tex
This paper relates smooth Lefschetz fibrations of 4-manifolds, the spherical braid group, and $\SL_2$-character varieties. Given a smooth Lefschetz fibration{\footnote{For us, we consider only Lefschetz fibration $\pi : M \to S^2$ with simple singularities, i.e., with at most one critical point in each fiber. We also work in the smooth category unless otherwise stated.}}
$\pi : M \to S^2$ whose generic fibers have genus $g$, let
\begin{align*}
	\Diff^+(\pi) &\coloneqq \{F \in \Diff^+(M) \st F \text{ takes fibers of } \pi \text{ to fibers of } \pi\}.
\end{align*}
The smooth mapping class group of $\pi$ is 
\begin{align*}
	\Mod(\pi) &\coloneqq \pi_0(\Diff^+(\pi)).
\end{align*}
Tracking where an element of $\Diff^+(\pi)$ takes the fibers gives a spherical braid{\footnotemark} monodromy representation
{\footnotetext{There are two distinct notions of the spherical braid group, one is $B_n(S^2) \coloneqq \pi_1(\operatorname{Conf}_n(S^2))$, the fundamental group of $n$-point configurations in $S^2$, the other is $\Mod(S^2, \{p_1,\ldots,p_n\})$, a marked mapping class group. These differ by an exact sequence $1 \to \ZZ/2\ZZ \to B_n(S^2) \to \Mod(S^2, \{p_1,\ldots,p_n\}) \to 1$ \cite[\S 9.1.4]{primer}. In this paper we restrict ourselves to the latter notion.}}
\begin{align*}
	\rho : \Mod(\pi) \to \Mod(S^2, \Delta_\pi),
\end{align*}
where $\Delta_\pi \subseteq S^2$ is the singular locus of $\pi$ and $\Mod(S^2,\Delta_\pi) = \pi_0(\Diff^+(S^2,\Delta))$. In this paper, we study the image
\begin{align}
	\Br(\pi) \coloneqq \im(\rho) < \Mod(S^2, \Delta_\pi), \label{defn:br-pi}
\end{align}
of the braid monodromy $\rho$. The subgroup $\Br(\pi) < \Mod(S^2, \Delta_\pi)$ consists of those braids that lift to a fiber-preserving diffeomorphism of $M$; we call these braids \textit{liftable braids} or \textit{liftable{\footnotemark} mapping classes}. Hence $\Br(\pi)$ associates to each Lefschetz fibration $\pi$ a subgroup of the spherical braid group. See \Cref{fig:two-braids} for an example of two braids whose marked points are the singular values of rational elliptic fibration $\pi : M \to S^2$. The top braid lifts while the bottom braid does not.
\footnotetext{In the context of covering maps $\widetilde{S} \to S$ between surfaces, the subgroup of $\Mod(S)$ consisting of liftable mapping classes--i.e., admitting a lift to $\widetilde{S}$--has been the subject of recent research. Much of this research focuses on the Putman--Wieland conjecture, which concerns the action of this subgroup on $H_1(\widetilde{S})$ \cite[Conjecture 1.2]{putman-wieland}.}

The main goal of this paper is to understand how large $\Br(\pi)$ is in comparison to the full spherical braid group $\Mod(S^2,\Delta)$. The coarsest such measure is the index of $\Br(\pi)$ inside $\Mod(S^2,\Delta)$. We give a general method for determining when the index $[\Mod(S^2,\Delta) : \Br(\pi)]$ is infinite which is applicable to a wide range of Lefschetz fibrations. Applications include \Cref{thm:inf-index}, \Cref{thm:holomorphic-genus-two}, and \Cref{thm:index-ell-disk-intro} (see below).

\begin{figure}
	\centering
	\subfloat[$r \in \Br(\pi)$]{
		\centering
		\hspace{0.5in}\includegraphics[scale=0.5]{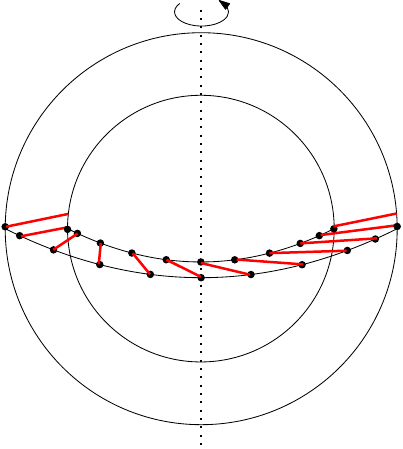} \hspace{0.5in}
		\includegraphics[width=0.22\textwidth]{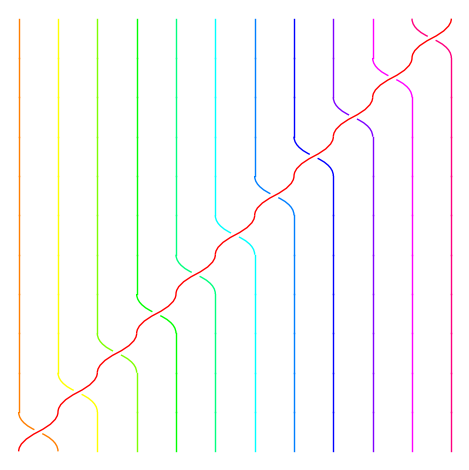}
	}\hspace{0.5in}
	\subfloat[$T_n \not\in \Br(\pi)$]{
		\centering
		\vspace{-0.1in}
		\includegraphics[scale=0.5]{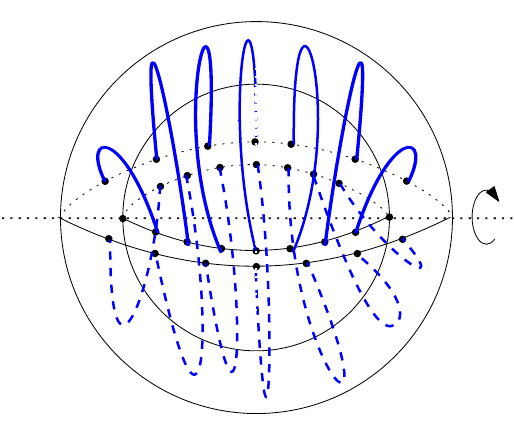} \hspace{0.6in}
		\includegraphics[width=0.15\textwidth]{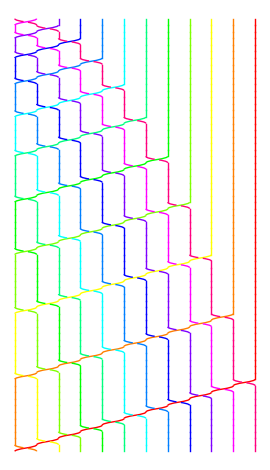}
	}
	\caption{The braids $r = \sigma_1\cdots \sigma_{n-1}$ given by rotation about the equator and the Garside half-twist $T_n = (\sigma_1\cdots \sigma_{n-1})(\sigma_1\cdots \sigma_{n-2})\cdots (\sigma_1\sigma_2)\sigma_1$ given by rotation about a longitude. $r$ lies in $\Br(\pi)$ while $T_n$ does not. In each case, the left figure depicts a spherical braid with its associated mapping class and the right figure depicts the corresponding element in the braid group $B_n$.}
	\label{fig:two-braids}
\end{figure}

\vspace{0.7\baselineskip}
\para{The Index of $\bm{\Br(\pi)}$} Our first result applies when $\pi$ is a self-fiber sum or when the fiber genus $g = 1$. For the definition of the fiber sum of two Lefschetz fibrations, we refer the reader to Gompf--Stipsicz \cite[\S 7.1]{gompf-stipsicz}
\begin{theorem}[Index of $\Br(\pi)$]\label{thm:inf-index}
	\noindent Let $\pi : M \to S^2$ be a nontrivial Lefschetz fibration of genus $g \geq 1$ with singular locus $\Delta_\pi \subseteq S^2$. Then
	\begin{enumerate}
		\item If $g = 1$ then $[\Mod(S^2, \Delta_\pi) : \Br(\pi)] = \infty$.
		\item If $\pi$ can be expressed as a self fiber-sum $\pi = p \fibsum p$ then $[\Mod(S^2, \Delta_\pi) : \Br(\pi)] = \infty$.
	\end{enumerate}
\end{theorem}
\vspace{-0.3cm}
\begin{remarks}\label{remark:inf-index}
	\hfill
	\begin{enumerate}[(i),leftmargin=0.75cm]
		\item  Part (1) of \Cref{thm:inf-index} can be deduced from Part (2). Explicitly, every genus $1$ Lefschetz fibration is expressed as a self-fiber sum of the rational elliptic surface given by the blowup $\Bl_{\{p_1,\ldots,p_9\}} \PP^2$ of $\PP^2$ along the intersection points of two cubics. The case of the rational elliptic surface itself can be addressed by a similar argument to that used in the proof of Part (2).\vspace{0.2cm}
		\item Part (1) \Cref{thm:inf-index}, the genus one case, can also be deduced from recent work of Lam, Landesman, and Litt classifying certain{\footnotemark} $\SL_2$-local systems on $(S^2, \Delta_\pi)$ with finite mapping class group orbit, see \Cref{sec:prelim} \cite{lam-lan-litt}. Their classification relies on non-abelian Hodge theory and is inspired by Katz's classification of rigid local systems. Our argument has the advantages of applying in genus $g \geq 2$ and being elementary; however, it relies heavily on our assumption that the singularities of our genus one fibrations are nodal. Lam--Landesman--Litt allows one to reproduce a variant of Part (1) for a wider variety of genus one fibrations allowing other types of singularities besides nodal cubics.
	\end{enumerate}
\end{remarks}
\footnotetext{Lam-Landesman-Litt's work only directly applies to those $\SL_2$-local systems where the monodromy about at least one puncture has infinite order. However, the local systems arising from genus one fibrations all have such infinite order monodromies.}

Using Chakiris' classification of holomorphic genus $2$ Lefschetz fibrations whose vanishing cycles are nonseparating \cite{chakiris}, we are also able to prove the following.
\begin{theorem}\label{thm:holomorphic-genus-two}
	Let $\pi : M \to S^2$ be a nontrivial holomorphic Lefschetz fibration with fibers of genus $g = 2$ and singular locus $\Delta$. If $\pi$ only has nonseparating vanishing cycles then 
	\[
		[\Mod(S^2,\Delta) : \Br(\pi)] = \infty.
	\]
\end{theorem}
As we will later see, given two Lefschetz fibrations $\pi$ and $\pi'$ such that $\Br(\pi)$ has infinite index in $\Mod(S^2,\Delta_\pi)$, the liftable braid group $\Br(\pi \fibsum \pi')$ of their fiber sum has infinite index in $\Mod(S^2,\Delta_{\pi \fibsum \pi'})$. Thus \Cref{thm:holomorphic-genus-two} is equivalent to proving that three explicit fibrations satisfy the infinite index property due to Chakiris' classification \cite{chakiris}.

\para{The Hurwitz Action} To prove \Cref{thm:inf-index,thm:holomorphic-genus-two}, we analyze the \textit{Hurwitz action} of $\Mod(S^2, \Delta_\pi)$ on the set
\begin{align*}
	\mathfrak{X}_g(S^2, \Delta_\pi) \coloneqq \Hom(\pi_1(S^2 \setminus \Delta_\pi), \Mod(\Sigma_g))/\Mod(\Sigma_g),
\end{align*}
where $\Mod(\Sigma_g)$ acts by conjugation and $\Mod(S^2,\Delta_\pi)$ acts by precomposition through its outer action on $\pi_1$. The orbit of a representation under this action is called its \textit{Hurwitz orbit}. In genus $g = 1$, there is a map from $\mathfrak{X}_1(S^2, \Delta_\pi)$ to the $\SL_2$-character variety:
\begin{align*}
	\mathfrak{X}_1(S^2,\Delta_\pi) \to \Hom(\pi_1(S^2 \setminus \Delta_\pi), \SL_2 \CC)\sslash \SL_2 \CC \eqqcolon \mathfrak{Y}(S^2,\Delta_\pi)(\CC),
\end{align*}
via the identification $\Mod(\Sigma_1) \cong \SL_2\ZZ$. For $\SL_2$-character varieties, the Hurwitz orbits are well-studied but not well-understood in general, with applications ranging from Diophantine equations via Markoff surfaces to algebraic solutions of ordinary differential equations. Of particular interest to the present paper are large orbits under the action on the mod $p$ points $\mathfrak{Y}(S^2,\Delta_\pi)(\mathbb{F}_p)$ as in \cite{markoff-triples-strong-approx} or finite orbits over $\CC$ as in \cite{lam-lan-litt}. Motivated by the relationship to classical character varieties in genus one, we call $\mathfrak{X}_g(S^2,\Delta_\pi)$ the \textit{modular character variety}. 

Given a Lefschetz fibration $\pi$, we consider its associated monodromy representation
\begin{align}
	\phi_\pi : \pi_1(S^2 \setminus \Delta_\pi, b) \to \Mod(\pi^{-1}(b)) \cong \Mod(\Sigma_g),\label{defn:mon-rep}
\end{align}
for $b \in S^2 \setminus \Delta_\pi$ and $\Sigma_g$ a genus $g$ connected oriented surface. Note that the monodromy representation $\phi_\pi$ is only well-defined up to conjugacy, due to the choice of identification of the fiber $\pi^{-1}(b)$ with $\Sigma_g$. Taking simple loops $\gamma_1,\ldots,\gamma_n$ about each puncture which whose product is nullhomotopic, we also obtain a factorization
\begin{align*}
	\phi_\pi(\gamma_1)\cdots \phi_\pi(\gamma_n) = \Id
\end{align*}
in the mapping class group. Picard--Lefschetz theory implies that $\phi_\pi(\gamma_i) \in \Mod(\Sigma_g)$ is a positive Dehn twist about some simple closed curve. The ordered $n$-tuple $(\phi_\pi(\gamma_i))$ is called the \textit{monodromy factorization of $\pi$}.

In this way, any Lefschetz fibration canonically determines a point $[\phi_\pi] \in \mathfrak{X}_g(S^2,\Delta_\pi)$. A liftable braid $[f] \in \Br(\pi) < \Mod(S^2,\Delta_\pi)$ must fix this point under the Hurwitz action because a lift $F \in \Diff^+(\pi)$ realizes a bundle isomorphism between $\pi$ and $f \circ \pi$ over $S^2 \setminus \Delta_\pi$. In general, the monodromy group is expected to be as large as possible, limited only by obvious geometric constraints. This slogan is referred to as ``big monodromy.'' Confirming this philosophy in our particular case, we realize $\Br(\pi) < \Mod(S^2,\Delta_\pi)$ as the stabilizer
\begin{equation}
	\Br(\pi) = \Stab_{\Mod(S^2,\Delta_\pi)} [\phi_\pi]\label{eq:br-pi-stab}
\end{equation}
in \Cref{sec:prelim}, with respect to the Hurwitz action. The identification \eqref{eq:br-pi-stab} can be inferred from work of Moishezon in genus $g = 1$ and Matsumuto in genus $g \geq 2$ (see \Cref{thm:moishezon-matsumuto} below) \cite{moishezon,matsumuto}. Their work also shows that two Lefschetz fibrations $\pi,\pi'$ are isomorphic if and only if $[\phi_\pi]$ and $[\phi_{\pi'}]$ lie in the same Hurwitz orbit. In this case $\Br(\pi)$ is conjugate to $\Br(\pi')$. 


\vspace{0.5\baselineskip}

\para{Finite Hurwitz Orbits over the Disk} We now describe our results on genus one Lefschetz fibrations $\pi : M \to D^2$ over the disk. The subgroup $\Br(\pi)$ is defined analogously to \eqref{defn:br-pi} above (see \Cref{sec:prelim} for details). In this case, $\Br(\pi)$ is realized as a subgroup of the classical braid group $\Mod(D^2, \Delta_\pi) \cong B_n$, where $n = \abs{\Delta_\pi}$ is the number of singular fibers. 

Let $q_n : M_n \to D^2$ be the genus one Lefschetz fibration with $n$ singular fibers and monodromy representation
\begin{align*}
	\phi_{q_n} : \pi_1(D^2 \setminus \text{ $n$ points}) &\to \Mod(\Sigma_1) \\
	\phi_{q_n}(\gamma_{2i+1}) = T_\alpha, &\quad \phi_{q_n}(\gamma_{2i}) = T_\beta, \addtocounter{equation}{1}\tag{\theequation}\label{eq:phiqn}
\end{align*}
where $\gamma_1,\ldots,\gamma_n$ are a standard system of generators for $\pi_1$ and $\alpha,\beta$ are simple closed curves in $\Sigma_1$ with geometric intersection number $i(\alpha,\beta) = 1$. Using the general criteria developed in \Cref{subsec:inf-index-general}, we prove the following theorem.
\begin{theorem}\label{thm:index-ell-disk-intro}
	Let $q_n : M_n \to D^2$ be the genus one Lefschetz fibration with monodromy $\phi_{q_n}$ given in \eqref{eq:phiqn}. If $n \geq 5$, then $[B_n : \Br(q_n)] = \infty$. The indices for $n \leq 4$ are
	\begin{align*}
		[B_2 : \Br(q_2)] = 3 && [B_3 : \Br(q_3)] = 8 && [B_4 : \Br(q_4)] = 27.
	\end{align*}
\end{theorem}
The Hurwitz orbits of the monodromy representation $\phi_{q_n}$ for $n = 3,4$ are displayed in \Cref{fig:braid-orbits-intro}, and the orbit for $n = 2$ is merely a triangle. The orbits of $\phi_{q_3}$ and $\phi_{q_4}$ are computed using a program developed by the author in SAGE/Python (see \eqref{eq:code}). The implemented algorithm is able to compute the Hurwitz orbit whenever it is finite and the fiber of the Lefschetz fibration $\pi : M \to D^2$ is a genus $g$ surface with one marked point. The input to the program is the monodromy representation of $\pi$. 

\begin{figure}[ht]
	\centering
	\subfloat[$n=4$]{
		\centering
		\includegraphics[scale=0.38]{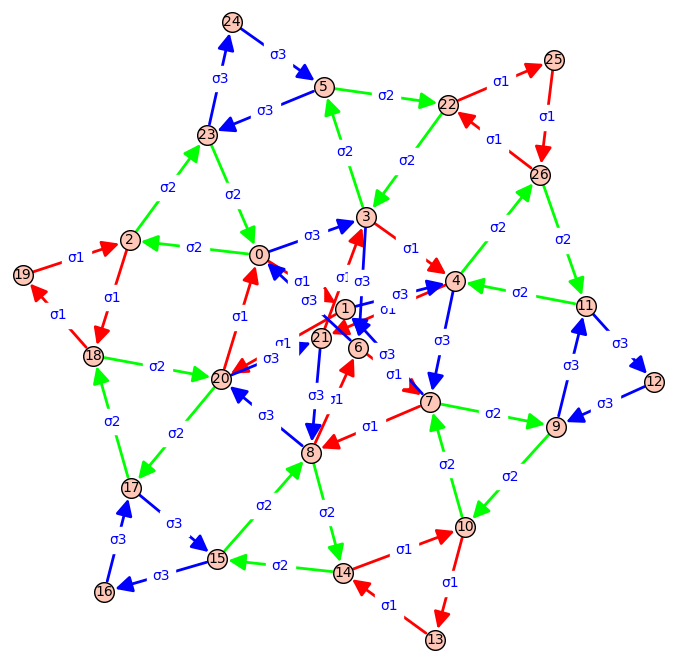}
	}\\ \vspace{-0.8in}
	\subfloat[Center of $n=4$]{
		\centering
		\includegraphics[scale=0.3]{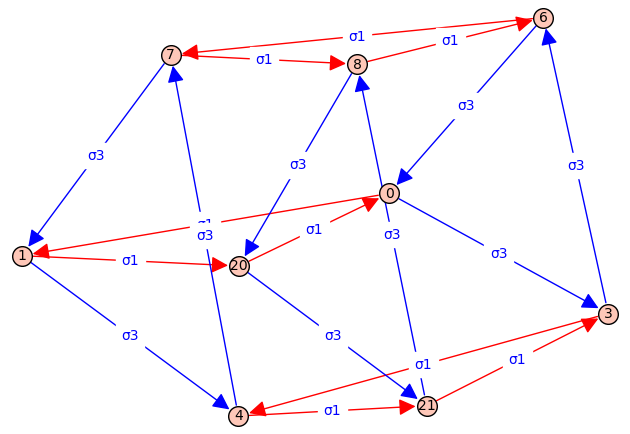}
	}\hspace{1in}
	\subfloat[$n=3$]{
		\centering
		\includegraphics[scale=0.45]{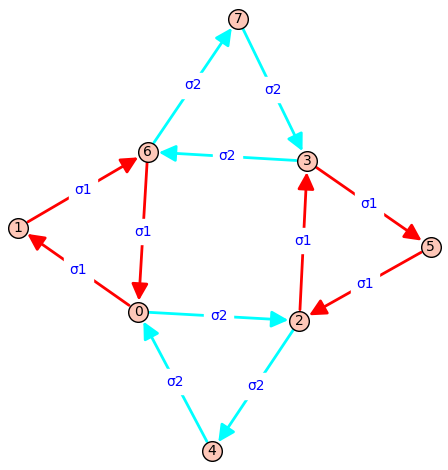}
	}
	\hspace{0.1in}

	\caption{Hurwitz orbits of $\phi_{q_n}$ for $n = 3$ and $4$. Vertex $0$ is $\phi_{q_n}$, and missing edges indicate representations fixed by the corresponding $\sigma_i$. When $n=3$ the orbit has size $8$, and when $n = 4$ it has size $27$. Figure (B) is the center portion of Figure (A).}
	\label{fig:braid-orbits-intro}
\end{figure}

For $\Br(q_3) < B_3$ and $\Br(q_4) < B_4$, we use GAP to compute generating sets for these subgroups. Call an arc between two singular fibers \textit{equinodal} if the vanishing cycles (in the sense of Picard--Lefschetz theory) for the singular fibers about these arcs are equal. The arc is called \textit{antinodal} if the vanishing cycles have geometric intersection number one. A direct computation shows that the half-twist about an equinodal arc lies in $\Br(\pi)$ and the third power of the half-twist about an antinodal arc lies in $\Br(\pi)$, for any genus $g$ fibration $\pi$. Our computation proves the following proposition.
\begin{proposition}\label{prop:fin-index-gen-set-intro}
	$\Br(q_3)$ and $\Br(q_4)$ are generated by half-twists about finitely many equinodal arcs along with the cubes of half-twists about finitely many antinodal arcs.
\end{proposition}

\vspace{0.5\baselineskip}

\para{Relation to Previous Work} For algebraic genus one fibrations $\pi : X \to \PP^1$, L\"{o}nne defines the \textit{bifurcation braid monodromy} of an algebraic family $\mathcal{X} \to \mathcal{B}$ of Lefschetz fibrations over $\PP^1$ containing $\pi$. Formally, such a family is equipped with a map $f : \mathcal{X} \to \mathcal{P}$, where $\mathcal{P}$ is a $\PP^1$-bundle over $\mathcal{B}$ so that the restriction $f_b : X_b \to \mathcal{P}_b$ is a genus one Lefschetz fibration and the restriction to some $b_0 \in \mathcal{B}$ is $\pi$ itself. In this setting, the \textit{bifurcation braid monodromy} of $\mathcal{X} \to \mathcal{B}$ is the composition
\begin{align*}
	\pi_1(\mathcal{B}) \to \Mod(\pi) \to \Br(\pi).
\end{align*}
L\"{o}nne studies in detail the group $\Br^{alg}(\pi)$ generated by the images of all such monodromies for all families $\mathcal{X} \to \mathcal{B}$ containing $\pi$ \cite{lonne-braid-monodromy}. L\"{o}nne shows that, when $\pi$ is a genus one fibration, there is some $b \in \Mod(S^2,\Delta)$ so that
\begin{align*}
	\Br^{alg}(\pi) = b\langle \sigma_{ij}^{m_{ij}} \mid m_{ij} = 1 \text{ if } i \equiv j \pmod 2 \text{ or } m_{ij} = 3 \text{ if } i \not\equiv j \pmod 2 \rangle b^{-1}
\end{align*}
where $\sigma_{ij}$ is a simple braid exchanging punctures $i,j$ on the sphere \cite[Main Theorem]{lonne-braid-monodromy}. L\"{o}nne's generating set for $\Br^{alg}(\pi)$ when $\pi$ is a genus one fibration agrees with the generating sets given in \Cref{prop:fin-index-gen-set-intro} for finite-index genus one fibrations over the disk. L\"{o}nne leaves open the question of whether $\Br^{alg}(\pi) = \Br(\pi)$ for genus one fibrations with nodal fibers. However, he does show that $\Br^{alg}(\pi) = \Br(\pi)$ for genus one fibrations with six nodal fibers and an odd number of singular fibers of type $I_0^\ast$.

Ito has previously examined the Hurwtiz orbits of representations from the free group $F_n$ to the braid group $B_3$ \cite{ito-fin-braid-hurwitz}. In particular, he provides criteria for these orbits to be finite, and upon applying the morphism $B_3 \to \SL_2\ZZ$, his examples induce genus one Lefschetz fibrations over the disk with finite Hurwitz orbit. E.g., the Hurwitz orbit of $q_3$, displayed here as \Cref{fig:braid-orbits-intro}, can be found as \cite[Fig. 6]{ito-fin-braid-hurwitz}.

\vspace{0.5\baselineskip}
\para{Further Questions} We hope this paper encourages further research into $\Br(\pi)$. To do so, we have compiled a list of questions. Given a Lefschetz fibration $\pi : M \to S^2$, we ask: 
\vspace{0.3cm}

\noindent \textbf{Question 1:} When is $\Br(\pi)$ finitely generated/finitely presented?


\noindent \textbf{Question 2:} Do there exist Lefschetz fibrations of genus $g$ over the sphere so that $\Br(\pi)$ is finite index?
\noindent \textbf{Question 3:} Can one classify the torsion elements of $\Br(\pi)$?
\vspace{0.3cm}

\noindent Resolving Question 1 in the affirmative for genus one Lefschetz fibrations would imply that $\Mod(\pi)$ is finitely generated (resp. finitely presented), as Farb--Looijenga have previously shown that in genus one $\ker(\Mod(\pi) \to \Mod(S^2,\Delta))$ is free of rank $\abs{\Delta}-4$ \cite[Theorem 1.2]{fl-smooth-mw}. Question 3 has prospective applications to the Nielsen Realization Problem. The author's previous paper addresses Question 3 in the case of elliptic fibrations and where the torsion elements in $\Br(\pi)$ have orders $\abs{\Delta}, \abs{\Delta} - 1,$ and $\abs{\Delta}-2$ \cite{braid-torsion-j}.

\vspace{0.5\baselineskip}
\para{Organization of the Paper} \Cref{sec:prelim} discusses the relationship between $\Br(\pi)$, i.e., the liftable mapping classes, and the monodromy representation of a Lefschetz fibration. We also give a proof of Part (1) of \Cref{thm:inf-index} based on the connection to Lam-Landesman-Litt's work and $\SL_2$-character varieties for genus one fibrations. We begin \Cref{sec:index} by proving \Cref{thm:inf-index} which correspond to the $g = 1$ and self-fiber sum cases. We then explore the exceptional finite-example examples over the disk in \Cref{subsec:finite-index}, proving \Cref{thm:index-ell-disk-intro} when $n \leq 4$ and \Cref{prop:fin-index-gen-set-intro}. In \Cref{subsec:inf-index-general}, we develop general criteria for showing $\Mod_\pi = \im(\Mod(\pi) \to \Mod(B,\Delta))$ has infinite index in $\Mod(B,\Delta)$ for Lefschetz fibrations over arbitrary base and prove \Cref{thm:index-ell-disk-intro} when $n \geq 5$ Using these general criteria, we prove \Cref{thm:holomorphic-genus-two} in \Cref{sec:holomorphic}, showing that holomorphic genus $g = 2$ Lefschetz fibrations with nonseparating vanishing cycles satisfy $[\Mod(S^2,\Delta) : \Br(\pi)] = \infty$.